\documentclass[12pt]{article}
\usepackage[british]{babel}
\usepackage{amsmath}
\usepackage{amsthm}
\usepackage{amsfonts}
\usepackage{amssymb}
\usepackage{fancyhdr}
\usepackage[all]{xy}           
\usepackage{graphicx}
\usepackage[left=3cm, right=3cm]{geometry}
\usepackage{enumerate}
\usepackage{cite}

\theoremstyle{plain}
\newtheorem{teo}{Theorem}[section]
\newtheorem{thm}[teo]{Theorem}

\newtheorem{lem}[teo]{Lemma}
\newtheorem{remark}[teo]{Remark}
\newtheorem{pro}[teo]{Proposition}

\def\M{\text{\rm M}}

\author{M\"{u}nevver P{\i}nar Ero\u{g}lu$^\dag$, Tsiu-Kwen Lee$^\flat$ and Jheng-Huei Lin$^\natural$}

\title{Certain functional identities on division rings of characteristic two}

\date{}

\begin{document}

\maketitle

\centerline {Department of Mathematics, Faculty of Science,$^\dag$}

\centerline {Dokuz Eyl\"{u}l University, \.{I}zmir, T\"{u}rkiye}

\centerline {mpinar.eroglu@deu.edu.tr$^\dag$}

\centerline {and}

\centerline {Department of Mathematics, National Taiwan
University${^\flat,}$${^\natural}$}

\centerline {Taipei, Taiwan}

\centerline {tklee@math.ntu.edu.tw$^\flat$; r01221012@ntu.edu.tw${^\natural}$}

\begin{abstract}\vskip4pt
\noindent Let $D$ be a noncommutative division ring. In a recent paper, Lee and Lin  proved that if $\text{\rm char}\, D\ne 2$, the only solution of additive maps $f, g$ on $D$
satisfying the identity $f(x) = x^n g(x^{-1})$ on $D\setminus \{0\}$ with $n\ne 2$ a positive integer is the trivial case, that is, $f=0$ and $g=0$.
Applying Hua's identity and the theory of functional and generalized polynomial identities, we give a complete solution of the same identity for any nonnegative integer $n$ if
$\text{\rm char}\, D=2$.
\end{abstract}

{ \hfill\break \noindent 2020 {\it Mathematics Subject Classification.}\ \  16R60, 16R50, 16K40.   \vskip4pt


 \vskip4pt

\noindent Corresponding author:\ Jheng-Huei Lin \vskip12pt

\section{Introduction}

Throughout, rings or algebras are associative with unity. Given a ring $R$, we denote by $R^\times$ the set of all units of $R$ and by $Z(R)$ the center of $R$.
The study of the identity $f(x) = x^n g(x^{-1})$ can be traced back to Halperin's question in 1963 of whether an additive map $f \colon \mathbb{R} \rightarrow \mathbb{R}$ satisfying the identity $f(x)=x^2 f(x^{-1})$ is continuous or not (see \cite{aczel1964}).
Later, Kurepa \cite{kurepa1964} studied the identity $f(t) = P(t)g(t^{-1})$, where $f,g \colon \mathbb{R} \rightarrow \mathbb{R}$ are nonzero additive maps and $P \colon \mathbb{R} \rightarrow \mathbb{R}$ is continuous with $P(1)=1$.
Several years later, Ng \cite{ng1987} completely characterized the identity $F(x) + M(x)G(1/x)=0$, where $F,G$ are additive and $M$ is multiplicative, on a field of characteristic not $2$.
In the same year, Vukman \cite{vukman1987} determined the identity $f(x)=-x^2 f(x^{-1})$ on a noncommutative division ring of characteristic not $2$.
In 2023, Dar and Jing \cite{dar2023} gave a complete characterization of the identity $f(x)=-x^2 g(x^{-1})$ for additive maps on noncommutative division rings of characteristic not $2$ and on matrix rings $\M_n(D)$, $n > 1$, over a noncommutative division ring $D$ with ${\text{\rm char}}\,D \ne 2,3$.
We state their result on noncommutative division rings as follows.

\begin{thm}(\cite[Theorem 1.4]{dar2023})
 Let $D$ be a division ring, which is not a field, with characteristic
 different from $2$ and let $f,  g\colon D\to D$ be additive maps satisfying the identity
 $f(x)+x^2g(x^{-1}) = 0$ for all $x\in D^{\times}$. Then $f(x) = xq$ and $g(x) =-xq$ for all $x\in D$,
 where $q\in D$ is a fixed element.
\end{thm}

Later, Catalano and Merch\'{a}n \cite{catalano2024} studied the identity
$f(x)=-x^n g(x^{-1})$
on a division ring $D^\times$, where $f , g$ are additive maps on $D$ and $n$ is a positive integer. They gave some calculations and obtained results for $n = 3$ and $n = 4$ with the assumption $\text{\rm char}\,D\ne 2, 3$.
Motivated by them, Luo, Chen, and Wang \cite{luo2023} generalized Catalano and Merch\'{a}n's results on division rings and on matrix rings over a division ring for $n > 2$ with characteristic zero or a prime integer $p>n$.
A few months later, Ferreira, Dantas, and Moraes \cite{ferreira2023} described the $n = 1$ case on division rings of characteristic not $2$ and on matrix rings over a division ring of characteristic not $2$. They also studied it on fields of characteristic $2$ and on matrix rings over a field of characteristic $2$.
In the recent paper \cite{lee2024}, the second and third authors completely solved it on noncommutative division rings of characteristic not $2$. They provided the following characterization.

\begin{thm} (\cite[Theorem 5.1]{lee2024})
  Let $D$ be a noncommutative division ring with ${\text{\rm char}}\,D \ne 2$. Let $f, g \colon D \rightarrow D$ be additive maps satisfying
  $f(x) = x^n g(x^{-1})$
  for all $x \in D^{\times}$, where $n\ne 2$ is a positive integer. Then $f = 0$ and $g = 0$.
\end{thm}

We remark that, on commutative fields $F$ of characteristic not $2$, the identity $f(x) = x^n g(x^{-1})$  for all $x \in F^{\times}$ was completely determined by Ng \cite{ng1987}.
Although the identity $f(x) = x^n g(x^{-1})$ for $n$ a positive integer has been completely solved on division rings of characteristic not $2$, it still remains unknown on division rings of characteristic $2$. In the paper we will prove the following main theorem.\vskip6pt

\begin{thm}  \label{thm6}
Let $D$ be a noncommutative division ring of characteristic $2$, and let $f, g \colon D\to D$ be additive maps. Suppose that  $f(x)=x^n g(x^{-1})$
for all $x\in D^{\times}$, where $n$ is a nonnegative integer. The following hold:
\begin{enumerate}[(i)]
  \item If $n = 2$, then $f = g$ and $f(x)=xf(1)$ for all $x\in D$;
  \item If $n \ne 2$, then $f =0$ and $g=0$.
\end{enumerate}
\end{thm}

We remark that Theorem \ref{thm6} is in general not true if $D$ is a field of characteristic $p>0$. See \cite[Example 4.3 (iii) and (iv)]{lee2024}.
Also, for $n=0$, Theorem \ref{thm6} is independent of the characteristic of $D$. That is, we have the following.

\begin{thm}  \label{thm9}
Let $D$ be a noncommutative division ring, and let $f, g \colon D\to D$ be additive maps. Suppose that  $f(x)=g(x^{-1})$
for all $x\in D^{\times}$. Then $f=0$ and $g=0$.
\end{thm}

The theorem will be proved in the last section.
Combining Theorems \ref{thm6} and \ref{thm9}, \cite[Theorem]{vukman1987}, \cite[Theorem 1.4]{dar2023} and \cite[Theorem 5.1]{lee2024}, the identity $f(x) = x^n g(x^{-1})$ for $n$ a nonnegative integer has been completely solved on noncommutative division rings. We state the conclusion here.

\begin{thm} \label{thm4}
Let $D$ be a noncommutative division ring, and let $n$ be a nonnegative integer. Suppose that $f,g \colon D\to D$ are additive maps satisfying
$f(x)=x^n g(x^{-1})$
for all $x\in D^{\times}$.  The following hold:
\begin{enumerate}[(i)]
  \item If $n = 2$, then $f = g$ and $f(x)=xf(1)$ for all $x\in D$;
  \item If $n \ne 2$, then  $f =0$ and $g=0$.
\end{enumerate}
\end{thm}

\section{Some preliminaries}
Throughout the paper, except for the last section, to keep the statements neat we always make the following assumption:\vskip6pt

{\it Let $D$ be a noncommutative division ring, $\text{\rm char}\,D=2$, and let $f, g \colon D\to D$ be additive maps.}
\vskip6pt

In order to prove Theorem \ref{thm6}, we begin with some preliminary results.
First, Theorem \ref{thm6} can be reduced to the case $f=g$.

\begin{lem}\label{lem3}
To prove Theorem \ref{thm6}, it suffices to assume that $f=g$.
\end{lem}

\begin{proof}
Clearly, we have $g(x) = x^n f(x^{-1})$ for all $x \in D^{\times}$. Thus
  \begin{eqnarray}
(f+g)(x) = x^n (f+g)(x^{-1})
\label{eq:5}
\end{eqnarray}
for all $x \in D^{\times}$.

(i)\ $n=2$:\ Assume that Theorem \ref{thm6} (i) is true for the case $f = g$.
By Eq.\eqref{eq:5}, we have $(f+g)(x) =x(f+g)(1)$ for all $x\in D$.
Note that $f(1)=g(1)$. Hence $f=g$, as desired.

(ii)\ $n\ne 2$:\
Assume that Theorem \ref{thm6} (ii) is true for the case $f = g$.
Then, by Eq.\eqref{eq:5}, we get $f+g=0$ and so $f=g$. Thus $f=0$ follows.
%
%
\end{proof}

\begin{lem} \label{lem4}
 Suppose that  $f(x)=x^n f(x^{-1})$
for all $x\in D^{\times}$, where $n$ is a positive integer.  Then
  \begin{eqnarray} \label{eq:1}
    f(a^2b) &=& ((1+ab)^n +(ab)^n +1 ) f(a) + a^n f(b)
  \end{eqnarray}
  for all $a, b \in D$ with $ab=ba$.
\end{lem}

\begin{proof}
  Let $a,b \in D$ with $ab \ne 0,1$. Then, by Hua's identity,
  \begin{eqnarray*}
     a+aba = (a^{-1}+(b^{-1}+a)^{-1})^{-1}.
  \end{eqnarray*}
  Thus
  \begin{eqnarray*}
   f(a+aba) &=& f((a^{-1}+(b^{-1}+a)^{-1})^{-1}) \nonumber \\
            &=& (a^{-1}+(b^{-1}+a)^{-1})^{-n} f(a^{-1}+(b^{-1}+a)^{-1}) \nonumber \\
            &=& (a^{-1}+(b^{-1}+a)^{-1})^{-n} \big(f(a^{-1}) + f((b^{-1}+a)^{-1})\big)  \\
            &=& (a+aba)^n \big(a^{-n}f(a)+(b^{-1}+a)^{-n}f(b^{-1}+a)\big) \nonumber \\
            &=& (a+aba)^n\big(a^{-n} + (b^{-1}+a)^{-n}\big)f(a) + (a+aba)^n (b^{-1}+a)^{-n}f(b^{-1}) \nonumber \\
            &=& (a+aba)^n (a^{-n} + (b^{-1}+a)^{-n} )f(a) + (a+aba)^n (b^{-1}+a)^{-n} b^{-n} f(b). \nonumber
  \end{eqnarray*}
If $ab = ba \ne 0,1$, then
\begin{eqnarray*}
    f(a+a^2b) &=& \big((1+ab)^n +(ab)^n\big)f(a) + a^n f(b).
\end{eqnarray*}
This implies that, for $a, b\in D$ with $ab=ba$, we have
\begin{eqnarray*}
f(a^2b) &=& \big((1+ab)^n +(ab)^n +1\big)f(a) + a^n f(b),
\end{eqnarray*}
as desired.
\end{proof}

Applying the standard Vandermonde argument, we have the following.

\begin{lem}\label{lem6}
Suppose that $\sum_{i=0}^na_i\lambda^i=0$ for all $\lambda\in Z(D)$, where $a_i\in D$ for all $i$. If $|Z(D)|>n+1$, then $a_i=0$ for all $i$.
\end{lem}

 Let $D\{X_1, X_2,\ldots,X_t\}$ denote the free product of $Z(D)$-algebra $D$ and the free algebra $Z(D)\{X_1, X_2,\ldots,X_t\}$ over $Z(D)$.  We refer the reader to \cite{martindale1969} or \cite{chuang1988} for details.
A generalized polynomial $f(X_1, X_2,\ldots,X_t)\in D\{X_1, X_2,\ldots,X_t\}$ is called a GPI for $D$ if $f(x_1, x_2,\ldots,x_t)=0$ for all $x_i\in D$.
We say that $D$ is a GPI-algebra if there exists a nonzero $f(X_1,\ldots,X_t)\in D\{X_1,\ldots,X_s\}$ such that $f(x_1,\ldots,x_t)=0$ for all $x_i\in D$.
In this case, we say that $f(X_1,\ldots,X_t)$ is a nontrivial GPI for $D$.
The following is a special case of \cite[Theorem 3]{martindale1969}.

\begin{thm} (Martindale 1969)\label{thm8}
Every division GPI-algebra is finite-dimensional over its center.
\end{thm}

We are now ready to prove the following lemma, which will be used in the sequel.

\begin{lem} \label{lem5}
Suppose that $f(x)=x^nf(x^{-1})$ for all $x\in D^{\times}$, where $n \geq 2$ is a positive integer. If $f(a^2)=0$ for all $a\in D$, then $f=0$.
\end{lem}

\begin{proof}
  Assume that $f(a^2)=0$ for all $a\in D$. Thus $f([a, b])=0$ for all $a, b\in D$.
  In particular,
  $f(ab)=f(ba)$ and $f((ab)^{-1})= f((ba)^{-1})$
  for all $a, b\in D^\times$.

  Let $a,b \in D$. By the definition of $f$, we have
  $$
  f(ab) = (ab)^n f((ab)^{-1})\ \ \text{\rm and}\ \ f(ba) = (ba)^n f((ba)^{-1}).
  $$
  If $f(ab) \ne 0$, then $(ab)^n = (ba)^n$, and so $[a, (ab)^n]=0$.

Suppose on the contrary that $f \ne 0$. Let $a \in D^{\times}$. There exists $b\in D$ such that $f(ab)\ne 0$.
  Let $x\in D$ and consider the following two cases.

Case 1:\ $f(a(b+x))\ne 0$. Then $[a, (a(b+x))^n]=0$.

Case 2:\ $f(a(b+x))=0$. Then $f(ax)=f(ab)\ne0$. So $[a, (ax)^n]=0$.

By the two cases, we have
 \begin{eqnarray}
  [a, (ax)^n][a, (a(b+x))^n]=0
\label{eq:6}
\end{eqnarray}
  for all $x\in D$. Assume that $a\notin Z(D)$. Note that
$$
  h(X):=[a, (aX)^n][a, (a(b+X))^n]
$$
is nonzero in $D\{X\}$. Indeed, it suffices to claim that $[a, (aX)^n][a, (aX)^n]\ne 0$ in $D\{X\}$. Otherwise, in particular, $[a, (ax)^n][a, (ax)^n]=0$
for all $x\in D$ and so $[a, (ax)^n]=0$ for all $x\in D$. According to \cite[Lemma 2]{lee1996}, $[a, ax]=0$ and hence $a[a, x]=0$ for all $x\in D$, forcing $a\in Z(D)$, a contradiction.

This implies that $D$ is a division GPI-algebra and thus $D$ is finite-dimensional over $Z(D)$ (see Theorem \ref{thm8}). Since $D$ is noncommutative, it follows from the well-known Wedderburn's theorem \cite{wedderburn1905} that $|Z(D)|=\infty$.

Replacing $x$ by $\lambda x$, where $\lambda\in Z(D)$, in Eq.\eqref{eq:6}, we get
 $[a, (\lambda ax)^n][a, (a(b+\lambda x))^n]=0$ and expand it as
 $$
 \sum_{i=0}^{2n}\lambda^ia_i=0,
 $$
 where $a_i\in D$. Note that $a_{2n}=[a, (ax)^n][a, (ax)^n]$. In view of Lemma \ref{lem6}, we get
 $[a, (ax)^n][a, (ax)^n]=0$ for all $x\in D$. Hence $a\in Z(D)$, a contradiction.
Thus we have $f = 0$, as desired.
\end{proof}

\section{Theorem \ref{thm6} with $n=0, 1, 2$}
We begin with studying the identity $f(x) = x^n g(x^{-1})$ with $n=0$.

\begin{thm} \label{thm7}
Assume that $f(x)=g(x^{-1})$
for all $x\in D^{\times}$. Then $f =0$ and $g = 0$.
\end{thm}

\begin{proof}
By Lemma \ref{lem3}, we may assume that $f=g$. Let $x\in D\setminus \{0, 1\}$.
Then
 \begin{eqnarray*}
 f(x(x+1))&=& f((x+1)^{-1}x^{-1})\nonumber\\
                 &=& f((x+1)^{-1}+x^{-1})\nonumber\\
                 &=& f(x+1)+f(x)\nonumber\\
                 &=& f(1).
 \end{eqnarray*}
That is,
\begin{eqnarray}
f(x^2+x+1)=0.
\label{eq:7}
\end{eqnarray}
Note that $x\in D\setminus \{0, 1\}$ if and only if $x^2\in D\setminus \{0, 1\}$.
Replacing $x$ by $x^2$ in Eq.\eqref{eq:7}, we get $f(x^4+x^2+1)=0$.
Hence
\begin{eqnarray}
f(x^4+x)=0
\label{eq:8}
\end{eqnarray}
for all $x\in D$.
Let $A$ denote the additive subgroup of $D$ generated by all elements $x^4+x$ for $x\in D$. By Eq.\eqref{eq:8}, we get $f(A)=0$.
In view of \cite[Theorem, p.98]{chuang1987}, either $x^4+x\in Z(D)$ for all $x\in D$ or $[D, D]\subseteq A$.

Suppose first that $x^4+x\in Z(D)$ for all $x\in D$. Then $D$ is a PI-algebra and hence $[D\colon Z(D)]<\infty$ by Posner's theorem \cite{posner1960}.
Since $D$ is not commutative, this implies that $|Z(D)|=\infty$. Thus $\lambda^4 x^4+\lambda x\in Z(D)$ for all $x\in D$ and all $\lambda\in Z(D)$. In view of Lemma \ref{lem6}, $x\in Z(D)$ for all $x\in D$. That is, $D$ is commutative, a contradiction.
Hence $[D, D]\subseteq A$ and so $f(xy)=f(yx)$ for all $x, y\in D$.

Let $y\in D^\times$, and let $x\in D\setminus \{0, 1, y, y+1\}$. Such an $x$ exists since $D$ is infinite. We compute
\begin{eqnarray}  \label{eq:9}
f(y)&=&f(x)+f(x+y)\nonumber\\
      &=&f(x^{-1}+(x+y)^{-1})\nonumber\\
      &=&f(x^{-1}((x+y)+x)(x+y)^{-1})\nonumber\\
      &=&f(x^{-1}y(x+y)^{-1})\nonumber\\
      &=&f((x+y)y^{-1}x)\nonumber\\
      &=&f(xy^{-1}x+x).\\
      \nonumber
\end{eqnarray}
Note that if $x\in D\setminus \{0, 1, y, y+1\}$, then $x+1\in D\setminus \{0, 1, y, y+1\}$. Replacing $x$ by $x+1$ in Eq.\eqref{eq:9}, we get
\begin{eqnarray*}
f(y)&=&f((x+1)y^{-1}(x+1)+x+1)\\
      &=&f(xy^{-1}x+[x, y^{-1}]+y^{-1}+x+1).\\
\end{eqnarray*}
Since $f(y)=f(xy^{-1}x+x)$ and $f([x, y^{-1}])=0$, we get $f(y^{-1})=f(1)$.
That is, $f(y)=f(1)$ for all $y\in D^\times$. Clearly, this implies that $f=0$, as desired.
\end{proof}

\begin{thm} \label{thm1}
Assume that $f(x)=x g(x^{-1})$
for all $x\in D^{\times}$. Then $f =0$ and $g = 0$.
\end{thm}

\begin{proof}
  According to Lemma \ref{lem3}, we may assume that $f = g$, and so it follows from Eq.\eqref{eq:1} that $f(a^2) = af(1)$ for all $a \in D$, implying $f([a,b]) = 0$ for all $a,b \in D$. By the definition of $f$, for any $a, b \in D^\times$, we have
  $$
  f(ab) = abf((ab)^{-1})\ \ \text{\rm and}\ \ f(ba) = baf((ba)^{-1}).
  $$

 Note that $ f(ab) = f(ba)$ and $f((ab)^{-1})=f((ba)^{-1})$.
  If $f(ab) \ne 0$, then $ab=ba$.
  That is, if $[a, b]\ne 0$ then $f(ab) = 0$.

  Let $a \notin Z(D)$. We claim that $f(a)=0$. Indeed, $[a, b]\ne 0$ for some $b \in D$ and hence $[a, 1+b]\ne 0$. Thus
  $$
  0 = f(a(1+b)) = f(a) + f(ab) = f(a).
  $$

Let $a \in Z(D)$. We choose $b \in D\setminus Z(D)$ and so $a+b \notin Z(D)$, implying
  $$
  0 = f(a+b) = f(a)+f(b) = f(a).
  $$
Therefore $f=0$, as desired.
\end{proof}

\begin{remark}\label{remark1}
{\rm The case $n = 1$ (i.e., the identity $f(x)=x g(x^{-1})$
for all $x\in D^{\times}$) has been described by Ferreira et al. if $\text{\rm char}\,D\ne 2$ (see \cite[Theorem 1.1]{ferreira2023}).
In fact, given $x\in D^{\times}$, $f(2x)=2x g((2x)^{-1})=x g(x^{-1})=f(x)$ and so $f(x)=0$, as desired.}
\end{remark}

The following characterizes the case $n = 2$.

\begin{thm} \label{thm2}
Assume that $f(x)=x^2 g(x^{-1})$
for all $x\in D^{\times}$. Then $f = g$ and $f(x)=xf(1)$ for all $x\in D$.
\end{thm}

\begin{proof}
 In view of Lemma \ref{lem3}, we may assume that $f=g$. According to Eq.\eqref{eq:1},
  \begin{eqnarray*}
    f(a^2b) &=& ((1+ab)^2 +(ab)^2 +1 ) f(a) + a^2 f(b)= a^2 f(b),
  \end{eqnarray*}
  for all $a,b \in D$ with $ab=ba$. In particular, $f(a^2)=a^2f(1)$ for all $a\in D$.
  Define $\widetilde{f}\colon D \to D$ by $\widetilde{f}(x) = f(x)-xf(1)$ for $x\in D$. Then $\widetilde{f}$ is an additive map satisfying
  $$
  \widetilde{f}(x)=x^2 \widetilde{f}(x^{-1})
  $$
  for all $x\in D$. Also, $\widetilde{f}(a^2) = 0$ for all $a\in D$. It follows from Lemma \ref{lem5} that $\widetilde{f} = 0$. Hence $f(x)=xf(1)$ for all $x\in D$.
\end{proof}

\section{Theorem \ref{thm6} with $n>2$}

Before giving the proof of Theorem \ref{thm6}, we always assume that\vskip6pt

{\it $D$ is a noncommutative division ring of characteristic $2$, and $f \colon D \to D$ is an additive map satisfying $f(x)=x^n f(x^{-1})$ for all $x\in D^{\times}$, where $n>2$.}\vskip6pt

 By Eq.\eqref{eq:1}, we have
\begin{eqnarray} \label{eq:2}
  f(a^2) = ((1+a)^n +a^n +1 ) f(a) + a^n f(1) = w(a)f(a) + a^n f(1)
\end{eqnarray}
for all $a\in D$, where $w(X) := (1+X)^n + X^n +1$. Then
\begin{eqnarray} \label{eq:3}
&&f([a,b]) \nonumber \\
&=& f((a+b)^2)+f(a^2)+f(b^2) \nonumber \\
&=& (w(a+b)+ w(a))f(a) + (w(a+b)+ w(b))f(b) + p(a,b)f(1)
\end{eqnarray}
for all $a,b\in D$, where $p(X,Y) := (X+Y)^n + X^n + Y^n \in D\{X,Y\}\setminus \{0\}$.

\begin{lem} \label{lem1}
  If $n$ is a positive power of $2$, then $f = 0$.
\end{lem}

\begin{proof}
 Clearly, $w(X) = 0$ in this case. It follows from Eq.\eqref{eq:2} and Eq.\eqref{eq:3} that
 $$
 f(a^2)=a^nf(1)\ \ \text{\rm and}\ \ f([a,b])= p(a,b)f(1)
 $$
  for all $a, b\in D$.
 Let $x,y,z\in D$. Note that
 $ [xy, z]+[yz, x]+[zx, y]=0$.
  We have
  \begin{eqnarray*}
    0 &=& f([xy, z]) + f([yz, x]) + f([zx, y]) \\
      &=& \big( p(xy, z) + p(yz, x) + p(zx, y) \big)f(1),
  \end{eqnarray*}
  implying either $f(1) = 0$ or $p(xy, z) + p(yz, x) + p(zx, y) = 0$ for all $x,y,z\in D$. We claim that $f(1) = 0$. Suppose on the contrary that $f(1) \ne 0$. Then
  $$
  p(xy, z) + p(yz, x) + p(zx, y) = 0
  $$
  for all $x,y,z\in D$. Thus
  \begin{eqnarray*}
 && p(XY, Z) + p(YZ, X) + p(ZX, Y)\\
 &=& (XY+Z)^n + (YZ+X)^n + (ZX+Y)^n \\
                                  &&+ (XY)^n + (YZ)^n + (ZX)^n + Z^n + X^n + Y^n
  \end{eqnarray*}
  is a nontrivial PI for $D$ since $n>2$ and it contains a nonzero term $(XY)^{n-1}Z$, implying $D$ is finite-dimensional over $Z(D)$ by Posner's theorem \cite{posner1960}. Since $D$ is noncommutative, it follows from Wedderburn's theorem \cite{wedderburn1905} that $|Z(D)|=\infty$. Recall that, for
$a, c\in D$, we have
  $$
  f([a, c])=p(a, c)f(1)=\big((a+c)^n+a^n+c^n\big)f(1).
  $$
 Let $\beta\in Z(D)$. Then $f([\beta a, c])=f([a, \beta c])$ and so
 $$
  (\beta a+c)^n+(\beta a)^n+c^n=(a+\beta c)^n+a^n+(\beta c)^n.
  $$
  Denote by $P_{i, j}(X, Y)$ the sum of all monic monomials with $X$  degree $i$ and $Y$ degree $j$.
  Then
  $$
  \sum_{i=1}^{n-1}P_{i, n-i}(c,a)\beta^{n-i} = \sum_{i=1}^{n-1}P_{i, n-i}(a,c)\beta^{n-i}.
  $$
Because $Z(D)$ is infinite, it follows from Lemma \ref{lem6} that
  $$
  P_{1, n-1}(c, a) = P_{1, n-1}(a, c).
  $$
Since $n>2$, applying Lemma \ref{lem6} once again we get $P_{1, n-1}(c, a)=0$ for all $a, c\in D$. So
  $$
  [a^n, c]=[a, P_{1, n-1}(c, a)]=0
  $$  for all $a, c\in D$.
  It follows from Kaplansky's theorem \cite[Theorem]{kaplansky1951} that $D$ is commutative, a contradiction.
  Hence $f(1) = 0$, implying $f(a^2) = 0$ for all $a \in D$. By Lemma \ref{lem6}, we have $f = 0$.
\end{proof}

A polynomial $P(X) \in K[X]$, where $K$ is a field, is called an additive polynomial if
$$
P(\alpha + \beta) = P(\alpha) + P(\beta)
$$
for all $\alpha, \beta \in K$. The following plays a key role in the proof below (see \cite[Proposition 1.1.5]{goss1996}).

\begin{pro}\label{pro1}
Let $K$ be an infinite field of characteristic $p>0$, and let $P(X) \in K[X]$ be an additive polynomial. Then $P(x)$ lies in the $K$-subspace of $K[X]$ generated by $X^{p^i}$ for $i=0, 1,\ldots$.
\end{pro}

 A map $f$ from $D$ into itself is called an elementary operator if there exist finitely
 many $a_i, b_i\in D$ such that $f(x)=\sum_{i}a_ixb_i$ for all $x\in D$.

\begin{lem} \label{lem2}
  If $n$ is not a positive power of $2$, then $f = 0$.
\end{lem}

\begin{proof}
 We write $n = 2^s k$, where $s \geq 0$ and $k > 1$ is odd.
  Then $w(X) \ne 0$. Applying Eq.\eqref{eq:3} and the Jacobi's identity, we expand the equation
  $$
  f([[x, y], z]) + f([[y, z], x]) + f([[z, x], y]) = 0
  $$
  to get an identity of the form
  \begin{eqnarray*}
  F(x, y, z)f(x)+G(x, y, z)f(y)+H(x, y, z)f(z)+ q(x,y,z) =0
  \end{eqnarray*}
  for all $x, y, z\in D$, where
  \begin{eqnarray*}
  F(X, Y, Z) &:=& \big(w([X, Y]+Z)+w([X, Y])\big)\big(w(X+Y)+w(X)\big)\\
               &&+ \big(w([Y, Z]+X)+w(X)\big)\\
               &&+\big(w([Z, X]+Y)+w([Z, X])\big)\big(w(Z+X)+w(X)\big)
  \end{eqnarray*}
  and $G(X, Y, Z), H(X,Y,Z), q(X,Y,Z) \in D\{X,Y,Z\}$.
  Note that $F(X, Y, Z)$ is nonzero in $D\{X,Y,Z\}$ since it contains the nonzero term $Z^{2^s (k-1)}Y^{2^s (k-1)}$.
  It follows from \cite[Corollary 2.4]{lee2024} that either $D$ is finite-dimensional over $Z(D)$ or $f$ is an elementary operator.
  We first claim that either $f=0$ or $D$ must be finite-dimensional over $Z(D)$.

Assume that $f$ is a nonzero elementary operator. There exist finitely many $c_j, d_j\in D$, $j=1,\ldots,s$, such that $f(x)=\sum_{j=1}^sc_jxd_j$ for all $x\in D$. We can choose $s$ minimal. So both $\{c_1,\ldots,c_s\}$ and $\{d_1,\ldots,d_s\}$ are $Z(D)$-independent.

 Case 1:\ $f(1) = 0$. Then, by Eq.\eqref{eq:2}, we get
  $$
  \sum_jc_jx^2d_j=\sum_j((1+x^{2^s})^k+x^{{2^s}k}+1)c_jxd_j
  $$
  for all $x\in D$, implying $D$ is a division GPI-algebra since $2^s(k-1)>1$. It follows from Theorem \ref{thm8} that $D$ is finite-dimensional over $Z(D)$.

 Case 2:\ $f(1) \ne 0$. Then, by \cite[Lemma 4.1(a)]{catalano2024}, we have
  $$
  ((1+x)^n + x^n +1)f(1) = 0
  $$
  for all $x\in D$, implying that
  $$
  (1+X)^n + X^n +1
  $$
  is a nontrivial PI for $D$. So it follows from Posner's theorem \cite{posner1960} that $D$ is also finite-dimensional over $Z(D)$.

Suppose on the contrary that $f\ne 0$. Then $D$ is finite-dimensional over $Z(D)$.
  Since $D$ is noncommutative, it follows from Jacobson's theorem \cite[Theorem 2, p. 183]{jacobson1964} (or see \cite[Theorem 13.11]{lam2001}) that $Z(D)$ is not algebraic over $\mathbb{Z}/2\mathbb{Z}$. In particular, $Z(D)$ is infinite.

  We next claim that $f(Z(D)) = 0$. Otherwise, let $\alpha \in Z(D)$ such that $f(\alpha) \ne 0$. By Eq.\eqref{eq:1}, we have
  $$
  f(\alpha^2 b) = ((1+\alpha b)^n +(\alpha b)^n +1 ) f(\alpha) + \alpha^n f(b)
  $$
  for all $b \in D$. Thus the map $b \mapsto (1+\alpha b)^n +(\alpha b)^n +1$ is additive, implying the nonzero polynomial
  $$
  P(X):=(1+\alpha X)^n +(\alpha X)^n +1
  $$
  is an additive polynomial in $Z(D)[X]$. Since $Z(D)$ is infinite, it follows from Proposition \ref{pro1} that $P(X)$ is a $Z(D)$-linear combination of monomials whose degrees are powers of $2$. Since $P(X)$ has a nonzero term $\alpha^{2^{\ell}(k-1)} X^{2^{\ell}(k-1)}$, $k-1$ must be a power of $2$. Write $k-1 = 2^t$ and so $n = 2^{s+t} + 2^{s}$. Then $P(X) = \alpha^{2^{s+t}} X^{2^{s+t}} + \alpha^{2^s} X^{2^s}$, and thus
  \begin{eqnarray*}
  Q(X,Y)&:=& P(X+Y)+P(X)+P(Y)\\
          &=& (\alpha^{2^{s+t}} (X+Y)^{2^{s+t}} + \alpha^{2^s}(X+Y)^{2^s}) \\
           &&+ ( \alpha^{2^{s+t}} X^{2^{s+t}} +\alpha^{2^s} X^{2^s} ) \\
           &&+ (\alpha^{2^{s+t}} Y^{2^{s+t}} +\alpha^{2^s} Y^{2^s})\\
           && \in Z(D)\{X,Y\}
  \end{eqnarray*}
  is a nontrivial PI for $D$.

  Let $F$ be a maximal subfield of $D$ containing $Z(D)$. Then $D \otimes_{Z(D)} F \cong \M_{r}(F)$, where $r = \sqrt{[D\colon Z(D)]}>1$. It is well-known that $D$ and $\M_{r}(F)$ satisfy the same PIs (see, for example, \cite[Corollary, p.64]{jacobson1975}). So $Q(X,Y)$ is also a PI for $\M_{r}(F)$. Since $Z(D)$ is not algebraic over $\mathbb{Z}/2\mathbb{Z}$, we can always choose $\beta \in Z(D)$ such that $\alpha^{2^{s+t}} \beta^{2^{s+t}} + \alpha^{2^s} \beta^{2^s} \ne 0$. However,
  $$
  Q(\beta e_{11},\beta e_{12}) = (\alpha^{2^{s+t}} \beta^{2^{s+t}} + \alpha^{2^s} \beta^{2^s} )e_{12} \ne 0,
  $$
  a contradiction. Hence $f(Z(D)) = 0$.

Let $b\in D$ with $f(b)\ne 0$. Then we have
  $$
  f(\alpha^2 b) = ((1+\alpha b)^n +(\alpha b)^n +1 ) f(\alpha) + \alpha^n f(b) = \alpha^n f(b)
  $$
  for all $\alpha \in Z(D)$. Replacing $\alpha$ by $\alpha + 1$, we have
  $$
  (\alpha+1)^n f(b) = f((\alpha+1)^2 b) = f(\alpha^2 b) + f(b) = (\alpha^n +1)f(b)
  $$
  for all $\alpha \in Z(D)$. Then $(\alpha+1)^n + \alpha^n +1 = 0$ for all $\alpha \in Z(D)$. Therefore $Z(D)$ is finite. Hence $D$ is a finite division ring, a contradiction. Hence $f = 0$, as desired.
\end{proof}\vskip4pt

\noindent {\bf Proof of Theorem \ref{thm6}.}

First, (i) follows from Theorem \ref{thm2}. Also, Theorems \ref{thm7} and \ref{thm1} show the cases $n = 0$ and $n=1$, respectively. For $n > 2$, by Lemma \ref{lem3}, we can assume that $f = g$. In this case, Lemmas \ref{lem1} and \ref{lem2} imply $f = 0$. Hence (ii) is proved.
\hfill $\square$
\vskip8pt

\section{Proof of Theorem \ref{thm9}}
In view of Theorem \ref{thm7}, we may assume $\text{\rm char}\,D\ne 2$.
Let $a,b \in D$ with $ab \ne 0,1$. Then, by Hua's identity,
\begin{eqnarray*}
	a-aba = (a^{-1}+(b^{-1}-a)^{-1})^{-1}.
\end{eqnarray*}
Thus, right-multiplying both sides by $a^{-1}$, we get
\begin{eqnarray*}
 1-ab&=& \big(1+a(b^{-1}-a)^{-1}\big)^{-1}.
\end{eqnarray*}
It means that
\begin{eqnarray*}
	(1-ab)^{-1}&=& 1+a(b^{-1}-a)^{-1}= 1+\big((ab)^{-1}-1\big)^{-1}
\end{eqnarray*}
and so
\begin{eqnarray*}
	1-ab&=& \big(1+((ab)^{-1}-1)^{-1}\big)^{-1}.
\end{eqnarray*}
Note that $f(1)=g(1)$. Thus
\begin{eqnarray*}
  f(1-ab) &=&g(1+((ab)^{-1}-1)^{-1}) \\
          &=&  g(1) +f((ab)^{-1}-1) \\
          &=& g(1) +g(ab) - f(1)\\
          &=& g(ab).
\end{eqnarray*}
Thus $(f+g)(ab) = f(1)$ for all $a,b \in D$ with $ab \ne 0,1$. Setting $b = 1$, we have $(f+g)(a) = f(1)$ for all $a \in D \setminus \{0,1\}$. By the additivity of $f+g$, we get $f=-g$.\vskip4pt

Thus $f(x)= - f(x^{-1})$ for all $x \in D^{\times}$. In particular, $f(1) = 0$.
Let $x\in D\setminus \{0, 1\}$.
Then
 \begin{eqnarray*}
 f(x(x+1))&=&- f((x+1)^{-1}x^{-1})\nonumber\\
                 &=&- f(x^{-1}-(x+1)^{-1})\nonumber\\
                 &=& f(x)-f(x+1)\nonumber\\
                 &=&0.
 \end{eqnarray*}
Hence $f(x^2+x)=0$ for all $x \in D$.
Let $A$ denote the kernel of $f$. Let $x, y\in D$. Clearly, we have $xy+yx\in A$. In particular, $2x^2\in A$ and hence $x^2\in A$ as $\text{\rm char}\,D\ne 2$.
Since $x^2+x\in A$, we conclude that $x\in A$. That is, $A=D$, i.e., $f=0$, as desired.
\hfill $\square$

\end{document}